\documentclass{amsart}
\usepackage{hyperref,enumerate}

\usepackage{soul,xcolor} 

\newtheorem{theorem}{Theorem}[section]
\newtheorem{lemma}[theorem]{Lemma}

\newtheorem{proposition}[theorem]{Proposition}
\newtheorem{corollary}[theorem]{Corollary}

\theoremstyle{definition}
\theoremstyle{acknowledgement} 
\newtheorem{definition}[theorem]{Definition}

\newtheorem{remark}[theorem]{Remark}
\numberwithin{equation}{section}

\title[Renormalized Area]{Renormalized area of minimal surfaces \\ in hyperbolic space}

\author{Matthew J. Gursky}
\address{Department of Mathematics \\
         University of Notre Dame\\
         Notre Dame, IN 46556}

\author{Khoi Nguyen}
\address{Department of Mathematics \\
         University of Notre Dame\\
         Notre Dame, IN 46556}

\author{Aaron J. Tyrrell}
\address{Department of Mathematics \\
         University of Notre Dame\\
         Notre Dame, IN 46556}

\begin{document}

\subjclass[2020]{Primary 53A10; Secondary 53C42, 58E12, 58J05.}

\keywords{Minimal surfaces, hyperbolic space, renormalized area, higher codimension, gap theorem}
\maketitle

\begin{abstract}  In this paper we consider minimal surfaces in hyperbolic space of arbitrary codimension that are critical for renormalized area.   We give two criteria which imply that $Y$ must be a totally geodesic disk.  One can be viewed as a 'gap' result for the renormalized area. 
\end{abstract}

\section{Introduction} 

In this article we consider two-dimensional minimal surfaces in hyperbolic space which are critical, in a sense described below, for the renormalized area.  To describe our setting in more detail, let $\mathbb{H}^{n+1}$ denote the upper half-space model of hyperbolic space: 
\begin{align*}
\mathbb{H}^{n+1} = \{ (x^1, \dots, x^n, r) \in \mathbb{R}^{n+1}_{+} \, : \, r > 0 \}, 
\end{align*}
where the hyperbolic metric is given by 
\begin{align} 
g_H = \dfrac{ (dx^1)^2 + \cdots + (dx^n)^2 + dr^2}{r^2}. 
\end{align}
Let $Y^2 \subset \mathbb{H}^{n+1}$ be the interior of a connected, compact, properly embedded minimal surface $\bar{Y} \subset \mathbb{H}^{n+1}\cup \{r=0 \}= \bar{\mathbb{H}}^{n+1}$, with $\gamma = \partial Y \subset \mathbb{R}^{n} = \{r=0\}$ an embedded closed curve. We will be interested in the case of codimension greater than one; i.e., $n \geq 3$.  

Boundary regularity theory for minimal surfaces of codimension one in hyperbolic space is highly developed.  For higher codimension much less is known; but see, for example, \cite{FLin}, \cite{marx2021variations}, \cite{jiang2026asymptotics}.   Since our interest is in the global variational theory of the renormalized area, we will assume enough regularity of the boundary curve and the minimal surface in order to justify our calculations. In particular, throughout the article we will assume $\gamma$ is $C^{4, \alpha}$, and $\bar{Y}$ is smooth in its interior and $C^{4,\alpha}$ up to the boundary.  

The renormalized area can be defined in a much more general setting than ours, and moreover for surfaces which are not necessarily minimal \cite{graham1999conformal}, \cite{AM}.  In our setting, however, the definition is more straightforward: given $\epsilon > 0$, let 
\begin{align*}
Y_{\epsilon} = \{ (x^1,\dots,x^n,r) \in Y \, : \, r > \epsilon \}. 
\end{align*}
The area of $Y_{\epsilon}$ has an expansion of the form 
\begin{align} \label{Ap}
\mbox{Area}(Y_{\epsilon}) = c_1 \epsilon^{-1} + c_0 + o(1), \ \ \ \epsilon \rightarrow 0.
\end{align}
The renormalized area $\mathcal{A}(Y)$ is the finite part of this expansion, i.e., $\mathcal{A}(Y) = c_0$.  Alexakis-Mazzeo \cite{AM} gave a closed formula for $\mathcal{A}(Y)$ (again in a more general setting than we consider here): 
\begin{align} \label{AYAM}
\mathcal{A}(Y) = - 2 \pi \chi(Y) - \frac{1}{2}\int_Y |B|^2 \, dA_h, 
\end{align} 
where $B$ is the trace-free second fundamental form of $Y$ and $dA_h$ the area form with respect to the induced metric $h = g_H \vert_Y$.  Since $Y$ is minimal, $B$ is just the second fundamental form.  Moreover, by conformal invariance the integral is convergent: 
\begin{align} \label{BB}
\int_Y |B|^2 \, dA_h = \int_Y |\bar{B}|^2 \, dA_{\bar{h}}, 
\end{align}
where $\bar{B}$ is the trace-free second fundamental form of $Y$ as a submanifold of $\mathbb{R}^{n+1}_{+}$ with the Euclidean metric $g_0$, and $\bar{h} = g_0 \vert_Y$.

From (\ref{AYAM}) we have the (non-obvious) inequality 
\begin{align} \label{Abelow}
\mathcal{A}(Y) \leq - 2 \pi \chi(Y), 
\end{align}
with equality if and only if $Y$ is totally geodesic.  By a subsequent result of Bernstein \cite{bernstein2022sharp}, the renormalized area satisfies   
\begin{align} \label{BernI}
\mathcal{A}(Y) \leq - 2 \pi,
\end{align}
and again equality holds if and only if $Y$ is totally geodesic. 

We are interested in surfaces which are critical for the renormalized area.  In Section \ref{Sec2} we give two possible interpretations of 'criticality.'  The first is that $Y$ is critical over variations of properly embedded surfaces $\{ Y_t \}_{-T < t < T}$ such that $Y_0 = Y$ and each $Y_t$ meets the boundary orthogonally; i.e., over surfaces which are {\em asymptotically minimal}.  Asymptotically minimal surfaces satisfy $|H|_{g_H} = O(r)$ as $r \to 0$, where $H$ is the mean curvature vector\footnote{Our convention will be that the mean curvature is the trace of the second fundamental form.}.  Alexakis-Mazzeo \cite[Proposition 3.1]{AM} showed that renormalized area is well defined for asymptotically minimal surfaces.   The second, and more restrictive meaning of criticality is that $Y$ is critical over variations by minimal surfaces $\{ Y_t \}_{-T < t < T}$ with $Y_0 = Y$.  

In the codimension-one case, Alexakis-Mazzeo showed that if $\gamma = \partial Y$ is connected and $Y$ satisfies a nondegeneracy condition, then $\gamma = \partial Y$ is a round circle and $Y$ is a totally geodesic disk (see Theorem 7.1 of \cite{AM}).   We remark that they proved this only assuming $\gamma$ is $C^{3,\alpha}$, since the regularity theory for minimal surfaces already gives that $Y$ is smooth in the interior and $C^{3,\alpha}$ up to the boundary.  

Our goal in this paper is to characterize critical minimal surfaces of higher codimension under various assumptions on the second fundamental form $A$ of $Y$.   Our first result assumes a pointwise bound on $A$:

\begin{theorem}  \label{Thm1}  Let $Y^2 \subset \mathbb{H}^{n+1}$ be a connected, properly embedded minimal surface satisfying the following regularity assumptions: \smallskip 

\begin{itemize} 

\item   $Y$ is smooth in its interior and $C^{4,\alpha}$ up to the boundary. \smallskip 

\item $\partial Y = \gamma \subset \mathbb{R}^n$ is an embedded, closed, $C^{4,\alpha}$-curve. \smallskip 

\end{itemize} 

\noindent Assume one of the following holds: \medskip 

\noindent $(i)$ $Y$ is critical for the renormalized area among asymptotically minimal surfaces (see Section \ref{Sec2} for definitions). \medskip 

\noindent $(ii)$ $Y$ is critical for the renormalized area among minimal surfaces and is nondegenerate (see Section \ref{Sec22} for definitions). \medskip 

\noindent If 
\begin{align} \label{Asmall} 
|A|^2 < 2
\end{align}
holds at each point of $Y$, then $\gamma = \partial Y$ is a circle and $Y$ is a totally geodesic disk.  
\end{theorem} 

\medskip 

Using an $\epsilon$-regularity argument, in Section \ref{SecER} we show that (\ref{Asmall}) follows from an $L^2$-smallness assumption on $A$.  Using the renormalized area formula of \cite{AM}, we can formulate this as a 'gap' result for the renormalized area:   

\medskip 

\begin{theorem}  \label{Thm2}  Let $Y^2 \subset \mathbb{H}^{n+1}$ be a connected, properly embedded minimal surface satisfying the regularity assumptions of Theorem \ref{Thm1}. There is a universal $\epsilon_0 > 0$ such that if 
\begin{align} \label{gap}
\mathcal{A}(Y) > - 2 \pi \chi(Y) - \epsilon_0
\end{align}
holds, and either \medskip 

\noindent $(i)$ $Y$ is critical for the renormalized area among asymptotically minimal surfaces; or \medskip 

\noindent $(ii)$ $Y$ is critical for the renormalized area among minimal surfaces and is nondegenerate, \medskip 

\noindent then $\gamma = \partial Y$ is a circle and $Y$ is a totally geodesic disk.  
\end{theorem}

\medskip 

\begin{remark} \label{DiskRemark}
The constant $\epsilon_0$ can be estimated (see Remark \ref{epRemark}), and does not depend on the codimension.  Once $\epsilon_0 > 0$ is small enough, the inequality of Bernstein (\ref{BernI}) already implies that $Y$ is topologically a disk. 
\end{remark}  

\medskip 

Critical minimal surfaces can be viewed as solutions of an overdetermined boundary value problem, in which the choice of boundary curve $\gamma$ is the Dirichlet data, and criticality amounts to a Neumann condition.  When $n=2$, Alexakis-Mazzeo classified (nondegenerate) critical surfaces by a clever maximum principle argument.  Our approach is necessarily quite different. 

The key to our proof is the fact that the norm of the mean curvature vector of $Y$ with respect to the Euclidean ({\em not} hyperbolic) metric satisfies a second order non-linear PDE.\footnote{This is the main reason we assume $C^{4,\alpha}$-regularity, instead of $C^{3,\alpha}$, as in Alexakis-Mazzeo.}   This fact amounts to a Bochner-type formula that uses the minimality of $Y$ (as a submanifold of $\mathbb{H}^{n+1}$) and the well known fact that $1/r$, where $r > 0$ is the last coordinate function, is an eigenfunction for the Laplace-Beltrami operator on $Y$. The assumption (\ref{Asmall}) allows us to apply a maximum principle argument to conclude that the mean curvature vector has constant norm.  From this it easily follows that $Y$ must be a totally geodesic disk in $\mathbb{H}^{n+1}$.  

A curious feature of our proof is that the criticality assumption is only used
to rule out the possibility that the maximum of the norm of the mean curvature vector occurs
at a point in \(\partial Y\).  Therefore, we obtain the following `strong maximum principle' for arbitrary
minimal surfaces:

\medskip 

\begin{theorem}  \label{Thm3}  Let $Y^2 \subset \mathbb{H}^{n+1}$ be a properly embedded minimal surface satisfying the regularity assumptions of Theorem \ref{Thm1}.  Assume either (\ref{Asmall}) or (\ref{gap}) holds.   

Let $\bar{H}$ denote the mean curvature vector of $Y \subset \mathbb{R}^{n+1}_{+}$. If $|\bar{H}|$ attains its maximum in the interior of $Y$, then $|\bar{H}|$ is constant, $\partial Y$ is a circle, and $Y$ is a totally geodesic disk. 
\end{theorem}

\medskip

\subsection{Organization of the paper}  We conclude with a brief outline of the paper.   In Section \ref{Sec2}, we review the asymptotics of minimal surfaces in hyperbolic space, and carry out the first variation calculation of the renormalized area.  We also introduce the notion of {\em nondegeneracy}, which will be crucial in passing from the first variation formula to the asymptotic behavior of the induced metric near $\partial Y$.   In Section \ref{SecPDE}, we show that the norm of the mean curvature vector of $Y$, viewed as a submanifold of Euclidean space, satisfies an elliptic PDE, and give the proof of Theorem \ref{Thm1}.  Finally, in Section \ref{SecER}, we use a Moser iteration argument to show that the $L^2$-smallness of the trace-free second fundamental form implies pointwise smallness, which combined with Theorem \ref{Thm1} proves Theorem \ref{Thm2}.

\medskip 

\subsection*{Acknowledgements.}  The authors would like to express their appreciation to Rafe Mazzeo for his interest in the project, and for sharing his insight on several occasions.    The first author was partially supported by the Simons Foundation Travel Support
for Mathematicians program, Award ID: SFI-MPS-TSM-00014122.

\section{Criticality }   \label{Sec2}

In this section we compute various asymptotic formulas for the minimal surface $Y$ near the boundary.  These will be used to give a first variation formula of the renormalized area.  

In \cite{marx2021variations} the author obtains a formula for the first variation of the renormalized area of a minimal submanifold $Y^m$ of an asymptotically hyperbolic Einstein manifold $X^{n+1}$ via Riesz regularization. The formula is given by the integral over the asymptotic boundary of the submanifold of the $m^{th}$ coefficient in the expansion of a particular function. In our case, $m=2$ and $X^{n+1}=\mathbb{H}^{n+1},$ so this coefficient can be explicitly identified. In \cite{marx2021variations} the author remarks that minimality of $Y_t$ is not required: as long as $Y_0$ is minimal, and some asymptotic properties hold, the formula remains valid.  We include this calculation for the sake of completeness.

For now let $Y^2 \subset \mathbb{H}^{n+1}\cup \{r=0 \}=: \bar{\mathbb{H}}^{n+1}$ be the interior of a compact properly embedded surface with $\gamma = \partial Y \subset \mathbb{R}^{n} = \{r=0\}$ an embedded closed curve. We will make it clear later when minimality is additionally imposed. Let $s$ be an arclength parameter for $\gamma$, and $\{e_1,...,e_{n-1} \}$ an orthonormal basis for the normal space of $\gamma \subset \mathbb{R}^n$ at a point $p = \gamma(0)$.  We use parallel transport with respect to the normal connection on $N\gamma$ to extend this to a local frame $\{ e_1(s),\dots,e_{n-1}(s) \}$ on $\gamma$ defined for $0 \leq s \leq b$.    

  Since $e_i$ is parallel with respect to the normal connection on $\gamma$, $(\nabla^{\mathbb{R}^n}_{\gamma'} e_i(s))^{\perp} = 0$, where $\nabla^{\mathbb{R}^n}$ is the Euclidean connection on $\mathbb{R}^n$ and $\perp$ is projection onto the normal space $N\gamma\vert_{\gamma(s)}$.  Letting primes denote differentiation with respect to $\frac{\partial}{\partial s}$, it follows that $0=(e_i')^{\perp}=e_i'-(e_i'\cdot \gamma')\gamma'$, where dot denotes the Euclidean inner product.  Consequently, $e_i'=(e_i'\cdot \gamma')\gamma'.$  Differentiating the formula $\gamma'\cdot e_i=0$, we get $\gamma''\cdot e_i=-\gamma'\cdot e_i' $. Denoting $k_i=-\gamma'\cdot e_i'$, we may summarize these relations as follows: 

\begin{align} \label{matrix}
\begin{bmatrix}
\gamma'' \\
e_1'\\
e_2'\\
\cdot \\
\cdot \\
\cdot \\
\cdot \\
e'_{n-1}
\end{bmatrix}=\begin{bmatrix}
0 & k_1 & k_2 & \cdot & \cdot & \cdot & k_{n-1} \\
-k_1 & 0 & 0 & \cdot & \cdot & \cdot & 0\\
-k_2 &0 &0 & \cdot & \cdot & \cdot & 0\\
\cdot & \cdot & \cdot & \cdot & \cdot & \cdot  & \cdot \\
\cdot & \cdot  & \cdot  &  \cdot & \cdot & \cdot & \cdot \\
\cdot & \cdot & \cdot &  \cdot & \cdot & \cdot  & \cdot \\
\cdot & \cdot & \cdot  &  \cdot & \cdot & \cdot & \cdot \\
-k_{n-1} &0 &0 &  \cdot & \cdot & \cdot  & 0 \\
\end{bmatrix}
\begin{bmatrix}
\gamma' \\
e_1\\
e_2\\
\cdot \\
\cdot \\
\cdot \\
\cdot \\
e_{n-1}\\
\end{bmatrix}.
\end{align}
\newline

Let $W$ be a tubular neighborhood of $\partial Y\times [0,\epsilon)$, and define $\psi\colon (0,b)\times B\times [0,\epsilon)\to W$ by \begin{equation}
\psi(s,x^1,...,x^{n-1},r)=\gamma(s)+x^1e_1(s)+x^{2}e_2(s)+ \dots +x^{n-1}e_{n-1}(s)+r\partial_r,
\end{equation}
where $B\subset \mathbb{R}^{n-1}$ is a ball around the origin small enough so that $V:=\psi((0,b)\times B\times [0,\epsilon))\subset W.$   The mapping $\psi$ induces a frame for the tangent bundle of $V$, given by 
\begin{align*}
\partial_{s}\psi&=\gamma '(s)[1-k_ix^i], \\
\partial_i\psi &=e_i(s),  \\
\partial_r\psi &=\partial_r.
\end{align*} 
With respect to this frame, for $1\leq i,j\leq n-1$, the components of the Euclidean metric $g_0$ are 
  \begin{align*}
  (g_0)_{ss}&=(1-k_ix^i)^2\\
(g_0)_{ij}&=\delta_{ij}\\
 (g_0)_{rr} &=1, 
\end{align*}
while all other components vanish. 

Define the slice
\[
Y_r\subset \mathbb R^n
\]
by
\[
\bar Y\cap \bigl(\mathbb R^n\times \{r\}\bigr)
=
Y_r\times \{r\}.
\]
Thus \(Y_0=\gamma\). If $\epsilon>0$ is small enough, then for each $r\leq \epsilon,$ \(Y_r\) lies in the tubular
neighborhood of \(\gamma\) given by $W\cap \{r=0\}$. Hence, shrinking $W$ if necessary, there is a unique section
\[
U_r\in \Gamma(N\gamma)
\]
such that
\[
Y_r
=
\left\{
\gamma(s)+U_r(s):s\in [0,L],
\right\}.
\]
where $L$ is the length of $\gamma$.  Equivalently, the surface \(\bar{Y}\cap V\) is parametrized by
\[
F:[0,L]\times [0,\epsilon)\longrightarrow
\mathbb R^n\times [0,\epsilon),
\qquad
F(s,r)=\gamma(s)+U_r(s)+r\partial_r.
\]
In particular,
\[
U_0=0.
\]

Observe that \(\bar{Y}\) meets $\{r=0\}$ orthogonally with respect to \(g_0\) if and only if $\partial_r\in T_{\gamma}\bar{Y}$ at \(r=0\).  Also, when $r=0$  \begin{equation}
F_*(c_1 \partial_s+ c_2 \partial_r)=c_1 \gamma'(s)+c_2 (\partial_r U_r\big|_{r=0}+\partial_r).
\end{equation}
Notice that if $F_*(c_1 \partial_s+ c_2 \partial_r) = c_3 \partial_r$, then using the fact that $\partial_r U_r$ is orthogonal to both $\gamma'(s)$ and $\partial_r$, we see that $c_1=0$ and $\partial_r U_r\big|_{r=0}=0.$   It follows that $\bar{Y}$ meets $\{r=0 \}$ orthogonally if and only if
\[
U_r=O(r^2).
\]
 
Let us write 
\begin{align} \label{Uform}
U_r(s)=u^i(s,r)e_i(s).
\end{align}
If we consider the set of points $p\in V$ satisfying the equation $x^i(p)=u^i(s(p),r(p))$, we see that these are parametrized by 
\begin{equation}
\psi(s,u^1(s,r),...,u^{n-1}(s,r),r)=\gamma(s)+u^i(s,r)e_i(s)+r\partial_r=\gamma(s)+U_r(s)+r\partial_r,
\end{equation}
 which is exactly $\bar{Y}\cap V$. It follows that 
 \begin{align*}
 \bar{Y}\cap V=\{p\in V\mspace{5mu} |\mspace{5mu} x^i(p)=u^i(s(p),r(p)), 1\leq i\leq n-1 \},
  \end{align*}
  hence locally $\bar{Y}$ is the graph of the functions $u^i.$ To summarize: $U_r$ gives an invariant description of the surface, whereas the functions $u^i$ are determined by the choice of a particular coordinate system. 
   
Let $\hat{s}$ and $\hat{r}$ be the induced coordinates on $\bar{Y}\cap V$ and observe that the tangent frame $\{\partial_{\hat{s}},\partial_{\hat{r}} \}$ induced by the graphing functions $u^i$ along $\bar{Y}\cap V$ are given by:
\begin{align*}
\partial_{\hat{s}}&= \partial_s+u^i_se_i, \\ 
\partial_{\hat{r}} &= \partial_r+u^i_re_i,
\end{align*}
where an $s$ or $r$ subscript corresponds to partial differentiation.
If $\bar{h}=(g_0)|_{\bar{Y}}$ is the metric on $Y$ induced by the Euclidean metric $g_0$, then the components of $\bar{h}$ in the $(r,s)$-coordinates are given by    
\begin{align} \label{hrs} \begin{split}
  \bar{h}_{ss}&=g_0(\partial_{\hat{s}},\partial_{\hat{s}}) =(1-k_iu^i)^2+\sum_i(u^i_s)^2, \\ 
\bar{h}_{rs}&=g_0(\partial_{\hat{s}},\partial_{\hat{r}})=\sum_{i}u^i_su^i_r, \\ 
 \bar{h}_{rr}&=g_0(\partial_{\hat{r}},\partial_{\hat{r}})=1+\sum_i(u^i_r)^2. 
 \end{split}
\end{align}

It is clear from these formulas that in order to understand the asymptotics of the components of $\bar{h}$, we need to understand the asymptotics of the $u^i.$  As shown in \cite{graham2017higher}, this can be achieved by studying a section of the normal bundle of $\partial Y\times [0,\epsilon)$ given by $\mathcal{M}(u)=\sum_{l}\mathcal{M}(u)_le_l(s)$, where $u = (u^1,\dots,u^{n - 1})$, whose components are given by 
\begin{align}\label{M}
\mathcal{M}(u)_\ell=&[r\partial_r-2+\frac{r(\log(\mathrm{det}\bar{h}))_r}{2}][\frac{(1-k_iu^i)^2+\sum_i(u_s^i)^2}{\mathrm{det}\bar{h}}u^\ell_r-\frac{\sum_{i}u_s^iu_r^i}{\mathrm{det}\bar{h}}u^\ell_s]\\\nonumber
&+r[\partial_s+\frac{(\log(\mathrm{det}\bar{h}))_s}{2}][\frac{1+\sum_i(u_{r}^{i})^2}{\mathrm{det}\bar{h}}u^{\ell}_{s}-\frac{\sum_{i}u_{s}^{i}u_{r}^{i}}{\mathrm{det}\bar{h}}u^{\ell}_{r}]\\\nonumber
&+\frac{r(1+\sum_i (u^{i}_{r})^2)}{\mathrm{det}\bar{h}}[1-k_ju^j]k_\ell,\quad 1 \leq \ell \leq n-1.
\end{align}
$\mathcal{M}(u)$ is related to the mean curvature vector $H_Y$ in the following way (see \cite{graham2017higher}): given a variation of $Y\cap V$ of the form $F_t(s,u(s,r),r)=(s,u_t(s,r),r)$, then
\begin{align*}
r^{-1}\mathcal{M}(u)_i\frac{du^i_t}{dt}|_{t=0}=g_H(H_Y,\frac{dF_t}{dt}|_{t=0}).
\end{align*}
In \cite{graham2017higher} the authors also show that $|H_Y|_{g_0}=O(r^k)$ if and only if $|\mathcal{M}(u)|_{g_0}=O(r^{k-1}).$ It follows that $|H_Y|_{g_H}=O(r)$ -- i.e., $Y$ is asymptotically minimal -- if and only if $|\mathcal{M}(u)|_{g_0}=O(r).$    

Setting (\ref{M}) equal to $O(r)$ and taking the limit as $r\to 0$ gives 
\begin{align*}
\partial_r u(s,0)=0,
\end{align*}
 which by (\ref{Uform}) is equivalent to 
 \begin{align}  \label{ortho} 
 \partial_r|_{r=0}\in T_{(s,0)}\bar{Y}.
  \end{align}
 It follows from (\ref{ortho}) that $Y$ is asymptotically minimal if and only if it meets $\{r=0\}$ at right angles. As shown in \cite{AM} these are exactly the surfaces which have a well-defined renormalized area.

Assuming each $u^i$ has a polyhomogeneous expansion, and plugging this into the minimality condition $\mathcal{M}(u) = 0$, one can recursively solve for the terms in the expansion.  Since the the minimal surface equation is degenerate elliptic, in general one could expect the presence of log terms.  However, for two-dimensional minimal surfaces in $\mathbb{H}^{n+1}$, it is not difficult to see that no log terms appear,\footnote{This is a partial justification for our regularity assumption that $Y$ is $C^{4,\alpha}$ up to the boundary.} and we can write  
\begin{equation}\label{expansion}
    u^{i}(s,r)=u^{i}_2(s)r^2+u^{i}_3(s)r^3+u^{i}_4(s)r^4 + O(r^5), \ \ 1\leq i\leq n-1.
\end{equation} 
One can compute  
\begin{align} \label{ui2}
u^{i}_2=\frac{k_i}{2},
\end{align}
but $u^{i}_3$ is not locally determined. 

Plugging (\ref{expansion}) into (\ref{Uform}), and defining 
\begin{align} \label{Ujdef}
U_{j}(s):=u^{i}_j(s)e_i(s),
\end{align}
we have 
\begin{equation}\label{expansion'}
    U_r(s)=U_2(s)r^2+U_3(s)r^3+U_4(s)r^4  + O(r^5), \ \ 
\end{equation} 
with
\begin{align} \label{U2form} 
U_2=\sum_i\frac{k_i}{2}e_i=\frac{\gamma''(s)}{2},
\end{align}
but $U_3$ is not locally determined.  

Plugging the expansion of $u^i$ into (\ref{hrs}) and using (\ref{ui2}) and (\ref{matrix}), we obtain the following asymptotic expansions for the components of the induced metric:

\begin{lemma}\label{hexpansion}
\begin{align} \label{hrs2} \begin{split}
  \bar{h}_{ss}&=1-|\gamma''(s)|^2r^2+O(r^3)\\
\bar{h}_{rs}&=O(r^3)\\
 \bar{h}_{rr}&=1+|\gamma''(s)|^2r^2+O(r^3). 
\end{split}
\end{align}
In fact, 
\begin{align} \label{hrr} 
\bar{h}_{rr} = 1+|\gamma''(s)|^2 r^2+ 6 \langle \gamma'', U_3 \rangle r^3 +O(r^4),
\end{align}
where $\langle \cdot , \cdot \rangle$ denotes the Euclidean inner product on the normal bundle. 
\end{lemma}

 In order to study variations of $\bar{Y}$ we need a frame for its normal bundle. To this end, in $\bar{Y}\cap V $ define $\nu_i(s,r)$ to be the projection of $e_i(s)$ onto $N_{(s,u(s,r),r)}\bar{Y}.$ Note that we may shrink $\epsilon$ if necessary to ensure that this gives us a frame for $N(\bar{Y}\cap V)$ given by  
\begin{align*}
\nu_i(s,r)&=e_i(s)-\bar{h}^{rr}g_0(e_i(s),\partial_{\hat{r}})\partial_{\hat{r}}-\bar{h}^{sr}g_0(e_i(s),\partial_{\hat{r}})\partial_{\hat{s}}\\&-\bar{h}^{sr}g_0(e_i(s),\partial_{\hat{s}})\partial_{\hat{r}}-\bar{h}^{ss}g_0(e_i(s),\partial_{\hat{s}})\partial_{\hat{s}}\\
&=e_i(s)-g_0(e_i(s),\partial_{\hat{r}})\partial_{\hat{r}}-g_0(e_i(s),\partial_{\hat{s}})\partial_{\hat{s}}+O(r^3)\partial_{\hat{\alpha}}\\
&=e_i(s)-[k_ir+3u_3^ir^2]\partial_{\hat{r}}-\frac{k_i'}{2}r^2\partial_{\hat{s}}+O(r^3)\partial_{\hat{\alpha}}.
\end{align*}

 Next let $F\colon (-T,T)\times \bar{Y} \to\bar{\mathbb{H}}^{n+1}$ be a normal variation of $\bar{Y}$. We will define $F_t\colon \bar{Y}\to \bar{Y}_t$ to be the map $[p\mapsto F(t,p)],$ and $F_0=id_{\bar{Y}}.$ Shrinking $T$ and $V$ if necessary we may assume that each $F_t$ is an embedding on $V$, which implies that there exist functions $\varphi_t^i(s,r)$ such that on $V$,
\begin{equation}\label{var}
F_t(s,r):=\gamma(s)+r\partial_r+u^i(s,r)e_i(s)+\varphi_t^i(s,r)\nu_i(s,r).
\end{equation}
If we fix $r=0$ then the map $[(t,s)\mapsto F_t(s,0)]$ is a variation of $\gamma.$ We will denote this map by $\gamma_t(s).$ Observe that $\partial_t\gamma_t|_{t=0}=\dot{\varphi}^i(s,0)e_i(s),$ where $\dot{}$ denotes the derivative with respect to $t$ at $t=0.$

\begin{lemma}\label{varasymptotics}
There exists $T>0$ such that if $Y_t$ is asymptotically minimal (i.e., $|H_{Y_t}|_{g_H}=O(r)$) for all $|t|<T$, then $\partial_r\varphi_t(s,0)=0$ for all $|t| < T$. 
\end{lemma}
\begin{proof}
We know from the discussion preceding Lemma \ref{hexpansion} that $\partial_r\in T\bar{Y}_t$ at $r=0;$ our strategy will be to prove that $\partial_r\in T\bar{Y}_t$ at $r=0$ if and only if $\partial_{r}\varphi_t(s,0)=0$ for all $t.$ At $r=0,$

\begin{align}
(F_t)_*(\partial_s)&= (1-\varphi_t^i(s,0)k_i)\gamma'(s)+\partial_s[\varphi_t^i(s,0)]e_i(s),
\end{align}

\begin{align}
(F_t)_*(\partial_r)&=\partial_r+\partial_r\varphi_t^i(s,0)\nu_i(s,0)+\varphi_t^i(s,r)\partial_r\nu_i(s,0)\\\nonumber
&=(1-\varphi^i_t(s,0) k_i)\partial_r+\partial_r \varphi_t^i(s,0)e_i(s).
\end{align} 
Since $\gamma', \partial_r$ and $e_i$ are orthonormal, $(F_t)_*(c_1 \partial_s+ c_2 \partial_r)=\partial_r$ if and only if the following equations hold:
\begin{align}
c_1 (1-\varphi_t^i(s,0)k_i)&=0\\ \label{2nd eqn}
c_2(1-\varphi^i_t(s,0) k_i)-1&=0\\ \label{3rd eqn}
c_1 \partial_s[\varphi_t^i(s,0)]+c_2\partial_r \varphi_t^i(s,0)&=0.
\end{align} 
If $c_1\neq 0$ then $\varphi_t^i(s,0)k_i=1,$ but $\varphi_0\equiv 0$ so for $0<T$ small enough it follows that $c_1=0.$ Observe that (\ref{2nd eqn}) implies $c_2\neq 0,$ hence by (\ref{3rd eqn}) we have the result. 
\end{proof}

Now that we have covered the preliminaries we may compute $\frac{d}{dt}|_{t=0}[\mathcal{A}(Y_t)]:$

\begin{proposition}  \label{1vProp} Let $\{Y_t\}_{|t|<T}$ be a family of properly embedded surfaces as in (\ref{var}) such that $Y=Y_0$ is minimal in $(\mathbb{H}^{n+1},g_H)$ and each $Y_t$ meets $\{r=0\}$ orthogonally. Then 
\begin{equation}
\frac{d}{dt}\bigg|_{t=0}[\mathcal{A}(Y_t)]=-3\int_{\gamma}\langle U_3,\partial_t \gamma_t|_{t=0}\rangle_{g_0}ds.
\end{equation}
\end{proposition}
\begin{proof}
Letting $h_t=F^*_t g_H$ and recalling the Riesz regularization characterization of the renormalized area (see Section 2.4 of \cite{marx2021variations}) we may write 

\begin{equation}
\mathcal{A}(Y_t)=\underset{z=0}{FP}\int_{Y_t} r^z dA_{Y_t},
\end{equation}
it follows that
\begin{align*}
\frac{d}{dt}\bigg|_{t=0}[\mathcal{A}(Y_t)]&=\frac{d}{dt}\bigg|_{t=0}\underset{z=0}{FP}\int_{Y_t} r^z dA_{Y_t}\\
&=\frac{d}{dt}\bigg|_{t=0}\underset{z=0}{FP}\int_{Y} F^*_t(r^z) F^*_t(dA_{Y_t})\\
&=\underset{z=0}{FP}\frac{d}{dt}\bigg|_{t=0}\int_{Y} F^*_t(r^z) F^*_t(dA_{Y_t}).
\end{align*}
By minimality we know $\frac{d}{dt}|_{t=0}F^*_t(dA_{Y_t})=0$ and by the chain rule we may write $\frac{d}{dt}|_{t=0}F^*_t(r^z)=zr^{z-1}d\hat{r}(\frac{dF_t}{dt}|_{t=0}).$ Using the product rule in the equation displayed above combined with the information in the previous sentence gives:

\begin{align}
\frac{d}{dt}\bigg|_{t=0}[\mathcal{A}(Y_t)]=\underset{z=0}{FP}\int_{Y} zr^{z-1}d\hat{r}\big(\frac{dF_t}{dt}\bigg|_{t=0}\big) dA_{Y}.
\end{align}
Let $f(s,r)=d\hat{r}\big(\frac{dF_t}{dt}|_{t=0}\big)\sqrt{\mathrm{det}\bar{h}}$ then expanding we get $f(s,r)=f(s,0)+f^{(1)}(s)r+f^{(2)}(s)r^2+..$.  If $r_0 > 0$ is sufficiently small, we find 
\begin{align*}
\int_{Y} zr^{z-1}d\hat{r}\big(\frac{dF_t}{dt}\bigg|_{t=0}\big) dA_{Y}&=zC+z\int_{0}^{r_0}\int_{\gamma} r^{z-3} d\hat{r}\bigg(\frac{dF_t}{dt}\bigg|_{t=0}\bigg)\sqrt{\mathrm{det}\bar{h}} ds dr\\
&=zC+z\int_{0}^{r_0}\int_{\gamma} r^{z-3} [f(s,0)+f^{(1)}(s)r+f^{(2)}(s)r^2+...] ds dr\\
&=zC+z\int_{\gamma} \bigg[ \frac{r^{z-2}}{z-2}f(s,0)\bigg|_{0}^{r_0}+\frac{r^{z-1}}{z-1}f^{(1)}(s)\bigg|_{0}^{r_0}+\frac{r^z}{z}f^{(2)}(s)\bigg|_{0}^{r_0}+... \bigg]ds
\end{align*}
Clearly the finite part as $z \to 0$ is $\int_{\gamma}f^{(2)}(s) ds.$ Next we will compute $f^{(2)}(s).$  After noting that $\sqrt{\mathrm{det}\bar{h}}=1+O(r^2)$ and 
\begin{align*}
&d\hat{r}(\frac{dF_t}{dt}\bigg|_{t=0})=\\&\frac{d}{dt}\bigg|_{t=0}g_0(-k_i\varphi_t^i(s,0)r\partial_{\hat{r}}-3\sum_i\varphi^i_{t}(s,0)u_3^ir^2\partial_{\hat{r}},\partial_{\hat{r}})+O(r^3)\\
&=-k_i\dot{\varphi}^i(s,0)r-k_i\partial_r\dot{\varphi}^i(s,0)r^2-3\sum_iu^i_3\dot{\varphi}^i(s,0)r^2+O(r^3)\\
&=-k_i\dot{\varphi}^i(s,0)r-3\sum_iu^i_3\dot{\varphi}^i(s,0)r^2+O(r^3)
\end{align*}
Note that we have used Lemma \ref{varasymptotics} to obtain the last equality in the calculation above.
Hence
\begin{align}
f^{(2)}(s)=-3\sum_iu_3^i(s)\dot{\varphi}^i(s,0)=-3\langle U_3(s),\partial_t\gamma_t(s)|_{t=0}\rangle_{g_0}.
\end{align}
\end{proof}

\begin{corollary}  \label{AHCor} 
If $\frac{d}{dt}|_{t=0}[\mathcal{A}(F_t(Y))]=0$ for all variations $F$ satisfying the hypotheses of Proposition \ref{1vProp}, then $U_3=0.$ 
\end{corollary}

\medskip

\subsection{The Jacobi operator and non-degeneracy} 
\label{Sec22} 

In this final subsection we introduce the notion of non-degeneracy appearing in the statements of Theorems \ref{Thm1} and \ref{Thm2}.  We begin with a brief recap of the relevant results in \cite{AM}. 

Suppose \(Y^2\subset \mathbb H^3\) is a properly
embedded minimal surface with asymptotic boundary
\(\gamma=\partial Y\).  The Jacobi operator associated to $Y$ is given by 
\[ L_Y = \Delta_Y +|A_Y|^2-2.
\]
Alexakis-Mazzeo showed that $L_Y$ is a uniformly degenerate elliptic operator with indicial roots \(-1\) and \(2\), and is Fredholm of 
index zero as an operator  
\[    L_Y:x^\mu \Lambda_0^{2,\alpha} \longrightarrow x^\mu \Lambda^{0,\alpha}_0, \ \ \ -1 < \mu < 2. 
\]
$Y$ is called {\em nondegenerate} if this map has no decaying Jacobi fields; i.e., no decaying kernel elements. 

Alexakis-Mazzeo also proved global structure theorems for the moduli space $\mathcal{M}_k(\mathbb{H}^3)$ of properly embedded $C^{3,\alpha}$-minimal submanifolds of $\mathbb{H}^3$ with genus $k$.  In particular, $\mathcal{M}_k(\mathbb{H}^3)$ is a Banach manifold, and the boundary map
\begin{align} \label{Pi}
\Pi : Y \mapsto \partial Y 
\end{align}
is proper (see \cite[Propositions 4.2 and 4.3]{AM}).  These global
properties give an integer-valued degree theory for the boundary map.

Turning to the case of higher codimension, the Jacobi operator is now a map on sections of the normal bundle  \(NY\), given by 
\begin{equation}
        L_YV
        =
        \Delta_Y^\perp V
        +
        \sum_{i,j=1}^2
        \big\langle A_Y(e_i,e_j),V\big\rangle_h A_Y(e_i,e_j)
        -2V ,
        \label{eq:higher-codim-jacobi}
\end{equation}
where \(\{e_1,e_2\}\) is any local orthonormal frame on \(TY\). The indicial roots are again $-1$ and $2$, hence 
\begin{align} \label{LFred}
 L_Y:r^\mu\Lambda^{2,\alpha}_0(Y;NY)
        \longrightarrow
        r^\mu\Lambda^{0,\alpha}_0(Y;NY), \ \ \ -1 < \mu < 2
\end{align}
is Fredholm.  Moreover, the same index argument as in \cite{AM} shows that the index is zero.   

\begin{definition}
We say that \(Y\) is \emph{nondegenerate} if \(Y\) has no nontrivial decaying normal Jacobi fields.
\end{definition}

In particular, if $Y$ is nondegenerate then the map in (\ref{LFred}) is an isomorphism.   We remark that in higher codimensions, the boundary map (\ref{Pi}) may fail to be proper, since sequences of minimal surfaces can have singular 
limits (see \cite{Fine}, Section 1.2).  For the argument below (based on Corollary 6.1 in \cite{AM}) we do not require properness of the boundary map, nor
any global degree theory.  We only use the local implicit-function-theorem statement
at the fixed nondegenerate surface \(Y\).

\begin{proposition}
\label{prop:nondegenerate-critical-implies-U3}
Assume that \(Y\) is nondegenerate. If \(Y\) is critical for the renormalized area among properly
embedded minimal surfaces, then
\[
        U_3 = 0.
\]
\end{proposition}

\begin{proof}  We will only give a sketch, since the idea is the same as \cite[Corollary 6.1]{AM}.   As we observed above, non-degeneracy implies that the map (\ref{LFred}) is an isomorphism. Therefore, the implicit function theorem gives
a local parametrization of minimal surfaces near \(Y\) by their
asymptotic boundary curves.  More precisely, after shrinking
neighborhoods, every sufficiently small section $\psi\in C^{4,\alpha}(\gamma;N\gamma)$ is realized as the normal component of the infinitesimal variation of
the boundary curve of some one-parameter family of nearby minimal
surfaces \(Y_t\), with \(Y_0=Y\).  (Here we have switched to $C^{4,\alpha}$, to be consistent with our regularity assumptions.) Writing
\[
        \gamma_t=\partial Y_t\subset \mathbb R^n=\{r=0\},
\]
this means
\begin{align} \label{gdot}
        \left(\partial_t\gamma_t\big|_{t=0}\right)^\perp=\psi,
\end{align}
where the projection is onto the Euclidean normal bundle $N\gamma$. 

For any such variation, the first variation formula for renormalized area in Proposition \ref{1vProp} gives 
\[
        \left.\frac{d}{dt}\right|_{t=0} A(Y_t)
        =
        -3\int_\gamma
        \left\langle
        U_3,
        \left(\partial_t\gamma_t\big|_{t=0}\right)^\perp
        \right\rangle_{g_0}\,ds .
\]
Since \(Y\) is critical among minimal surfaces, the left-hand side
vanishes.  Using (\ref{gdot}), we obtain
\[
        \int_\gamma \langle U_3,\psi\rangle_{g_0}\,ds=0
\]
for all \(\psi\in C^{4,\alpha}(\gamma;N\gamma)\). It follows that \(U_3\equiv 0\) on \(\gamma\).
\end{proof}

\bigskip 
 
\section{A PDE for the mean curvature} \label{SecPDE} 

 In this section we derive two PDEs that will be used in our analysis, and use these to give a proof of Theorem \ref{Thm1}.  The first is a PDE for $r\vert_{Y}$, the last coordinate function of $\mathbb{R}^{n+1}$.  The second is a PDE satisfied by the norm of the mean curvature vector $\bar{H}$ of $\bar{Y}$ as a submanifold of $\mathbb{R}^{n+1}$.    
 
 We begin with a preliminary lemma: 

\begin{lemma} \label{HbarLemma}  The mean curvature vector $\bar{H}$ of $\bar{Y}$ with respect to the Euclidean metric $g_0$ satisfies 
\begin{align} \label{mc2}
|\bar{H}|^2=\frac{4}{r^2}\left( 1-|\nabla_{\bar{h}} r|_{\bar{h}}^2\right).
\end{align}
\end{lemma} 
    \begin{proof} 
We first observe that $\bar{H}$ is given by 
\begin{align}\label{mc1}
\bar{H}=-\frac{2}{r}\nabla_{g_0}^{\perp}r,
\end{align} 
where $\perp$ denotes the projection onto the normal bundle of $Y$.  To see this, recall the conformal transformation law for the second fundamental form: Since $g_H = r^{-2}g_0$,  we have  
\begin{align*} 
A=\bar{A}+(\nabla_{g_0} \log r)^{\perp}\bar{h},
\end{align*}
where $A$ is the second fundamental form of $Y$ with respect to $g_H$, and $\bar{A}$ is with respect to $g_0$.  Taking the trace of both sides with respect to $h$, we get 
\begin{align*}
H=r^2\bar{H}+ 2 r \nabla_{g_0}^{\perp}r.
\end{align*}
Since $Y$ is minimal with respect to $g_H$, $H=0$ and (\ref{mc1}) follows.

With respect to the Euclidean metric, $|\nabla_{g_0} r|_{g_0}^2 = 1$. On the other hand, we have the orthogonal decomposition 
    \begin{equation}\label{mc3}
    \nabla_{g_0} r=\nabla_{\bar{h}} r+ \nabla_{g_0}^{\perp}r,
    \end{equation} 
    so that 
    \begin{equation}\label{mc3}
    1=|\nabla_{g_0} r|_{g_0}^2=|\nabla_{\bar{h}}  r|_{\bar{h}}^2+|\nabla^{\perp}r|_{g_0^{\perp}}^2,
    \end{equation}
    where $g_0^{\perp}$ is the metric induced on the normal bundle by $g_0$. It follows from (\ref{mc1}) that $|\bar{H}|_{g_0^{\perp}}^2=\dfrac{4}{r^2}|\nabla^{\perp} r|_{g_0^{\perp}}^2.$  Substituting this into (\ref{mc3}), we arrive at (\ref{mc2}).
\end{proof}
  
\medskip 

\begin{lemma}  \label{rPDELemma}  $r : Y \rightarrow \mathbb{R}$ satisfies 
\begin{align} \label{rPDE1} 
\Delta_{\bar{h}} r = 2 r^{-1} \left( |\nabla_{\bar{h}} r|_{\bar{h}}^2 - 1 \right),
\end{align}
where $\bar{h} = g_0\vert_{Y}$ and $g_0$ is the Euclidean metric. 
\end{lemma}

\begin{proof} Since $h = g_H \vert_Y = r^{-2} \bar{h} := e^{2w} \bar{h}$, the conformal transformation formula for the Gauss curvature implies  
\begin{align} \label{GC}
\Delta_{\bar{h}} w + K_h e^{2w} = K_{\bar{h}}, 
\end{align} 
where $K_h$ and $K_{\bar{h}}$ are the Gauss curvature of $h$ and $\bar{h}$, respectively.  Note that 
\begin{align} \label{w1} \begin{split} 
\Delta_{\bar{h}} w &= - \Delta_{\bar{h}} \log r \\
&= - r^{-1} \Delta_{\bar{h}} r + r^{-2} |\nabla_{\bar{h}} r|_{\bar{h}}^2, 
\end{split} 
\end{align}
hence (\ref{GC}) can also be written 
\begin{align} \label{GC2} 
- r^{-1} \Delta_{\bar{h}} r + r^{-2} |\nabla_{\bar{h}} r|_{\bar{h}}^2 + r^{-2} K_h = K_{\bar{h}}. 
\end{align}

Since $Y \subset \mathbb{H}^{n+1}$ is minimal, the Gauss curvature equation reads 
\begin{align} \label{GCE} 
K_h = - 1 - \frac{1}{2} |B|^2,
\end{align}
where $B$ is the trace-free second fundamental form of $Y$ with respect to $g_H$, and from now on the norm of any section is the natural norm induced by the associated ambient metric (in this case $g_H$).  Also, with respect to the Euclidean metric, 
\begin{align} \label{GCEb}
K_{\bar{h}} = - \frac{1}{2} |\bar{B}|^2 + \frac{1}{4} |\bar{H}|^2, 
\end{align}
where $\bar{B}$ is the trace-free second fundamental form, and  $\bar{H}$ is the mean curvature vector of $\bar{Y}$.  Substituting (\ref{GCE}) and (\ref{GCEb}) into (\ref{GC2}), and using the fact that $|B|^2 = r^2|\bar{B}|^2$, we get 
\begin{align} \label{GC3} 
- r^{-1} \Delta_{\bar{h}} r + r^{-2} |\nabla_{\bar{h}} r|_{\bar{h}}^2 + \left( - r^{-2} -\frac{1}{2} |\bar{B}|^2 \right) = - \frac{1}{2} |\bar{B}|^2 + \frac{1}{4} |\bar{H}|^2, 
\end{align}
hence 
\begin{align*}
\Delta_{\bar{h}} r = - \frac{1}{4} r |\bar{H}|^2 + r^{-1} \left( |\nabla_{\bar{h}} r|_{\bar{h}}^2 - 1 \right).
\end{align*}
Then (\ref{rPDE1}) follows from Lemma \ref{HbarLemma}.
\end{proof}

To simplify notation, we denote 
\begin{align} \label{etadef}
\eta = \frac{1}{4} |\bar{H}|_{\bar{h}}^2. 
\end{align}
Also, define $u : Y \rightarrow \mathbb{R}$ by
\begin{align} \label{udef}
u = \frac{1}{r}
\end{align}
Since 
\begin{align*}
\Delta_{h} = r^2 \Delta_{\bar{h}}, 
\end{align*}
it follows from (\ref{rPDE1}) that in $Y$, $u$ satisfies 
\begin{align} \label{uPDE2}
\Delta_h u = 2 u. 
\end{align} 
Also, by Lemma \ref{HbarLemma} we can also express $\eta$ as 
\begin{align} \label{etau}
\eta = u^2 - |\nabla u|^2, 
\end{align} 
where the norm and gradient on the right are with respect to $h$. 

Let $\mathcal{C} \subset Y$ be the set of critical points of $r : Y \rightarrow \mathbb{R}^{+}$.   By Lemma \ref{hexpansion}, for $r > 0$ small we have 
\begin{align*}
|\nabla_{\bar{h}} r|_{\bar{h}}^2 = 1 + O(r^2), 
\end{align*}
hence $\mathcal{C} \subset Y \cap \{ r > \delta_0 \}$ for some $\delta_0 > 0$.  Of course, the set of critical points of $u$ is also given by $\mathcal{C}$.

\begin{lemma} \label{betaLemma}  Suppose $Y$ is critical in either of the senses of Theorem \ref{Thm1}.  Then near $\partial Y$, $\eta = \eta(s,r)$ satisfies 
\begin{align} \label{etar}
\eta = |\gamma''(s)|^2 + O(r^2). 
\end{align}
\end{lemma}

\begin{proof}  This follows from Lemma \ref{hexpansion}, especially (\ref{hrr}), which in the case of criticality reads 
\begin{align*} 
\bar{h}_{rr} = 1+|\gamma''(s)|^2 r^2 +O(r^4).
\end{align*}
Then using  
\begin{align*}
&\eta = r^{-2} \left( 1-|\nabla_{\bar{h}} r|_{\bar{h}}^2\right), \\
&|\nabla_{\bar{h}} r|_{\bar{h}}^2 = \bar{h}^{rr}, 
\end{align*}
we easily obtain (\ref{etar}). 
\end{proof}

\medskip 

A key observation is that $\eta$ satisfies a PDE in $Y$:  \medskip

\begin{proposition}  \label{PDEProp}  $(i)$  On $Y$, we have 
\begin{align} \label{etaPDE1}
\Delta \eta = - 2 |\mathring{\nabla}^2 u|^2 + |B|^2 |\nabla u|^2, 
\end{align} 
where $\mathring{\nabla}^2 u = \nabla^2 u - \frac{1}{2}(\Delta u)h$ is the trace-free Hessian of $u$ with respect to $h$, all norms and covariant derivatives are with respect to $h$, and $B$ is the trace-free second fundamental form of $Y \subset \mathbb{H}^{n+1}$. 

$(ii)$ In $Y \setminus \mathcal{C}$, $\eta$ satisfies 
\begin{align} \label{etaPDE2}
\Delta \eta = - \dfrac{|\nabla \eta|^2}{|\nabla u|^2} + | B | ^ 2 |\nabla u | ^ 2. 
\end{align} 

$(iii)$  In $Y \setminus \mathcal{C}$, with respect to $\bar{h}$, $\eta$ satisfies 
\begin{align} \label{etaPDE3} 
\Delta_{\bar{h}} \eta = - r^2 \dfrac{ |\nabla_{\bar{h}} \eta|_{\bar{h}}^2}{|\nabla_{\bar{h}} r|_{\bar{h}}^2} + r^{-2} |\bar{B}|^2  |\nabla_{\bar{h}} r|_{\bar{h}}^2. 
\end{align} 
\end{proposition}

\begin{proof}  To prove $(i)$, we begin by differentiating both sides of (\ref{etau}):  
\begin{align} \label{P1} 
\nabla \eta &= 2 u \nabla u - \nabla |\nabla u|^2.
\end{align} 
Differentiating again and using (\ref{uPDE2}) gives 
\begin{align} \label{P2} \begin{split} 
\Delta \eta &= 2 u \Delta u + 2 |\nabla u|^2 - \Delta |\nabla u|^2 \\
&= 4 u^2 + 2 |\nabla u|^2 - \Delta |\nabla u|^2. 
\end{split} 
\end{align}
By the Bochner formula and (\ref{uPDE2}), 
\begin{align} \label{P3} \begin{split}
\Delta |\nabla u|^2 &= 2 |\nabla^2 u|^2 + 2 K |\nabla u|^2 + 2 \langle \nabla u , \nabla ( \Delta u) \rangle \\
&= 2 |\nabla^2 u|^2 + 2 K |\nabla u|^2 + 4 |\nabla u|^2. 
\end{split}
\end{align}
Using (\ref{GCE}) and the fact that 
\begin{align*}
|\nabla^2 u|^2 &= |\mathring{\nabla}^2 u|^2 + \frac{1}{2}(\Delta u)^2 \\
&= |\mathring{\nabla}^2 u|^2 + 2 u^2,
\end{align*}
we obtain  
\begin{align} \label{P4} \begin{split}
\Delta |\nabla u|^2 &= 2 |\mathring{\nabla}^2 u|^2 + 4u^2  - |B|^2 |\nabla u|^2 + 2 |\nabla u|^2. 
\end{split}
\end{align}
Substituting this into (\ref{P2}), we arrive at (\ref{etaPDE1}). 

To prove (\ref{etaPDE2}), we return to (\ref{P1}) and observe that 
\begin{align} \label{P5} \begin{split}
\nabla \eta &= 2 u \nabla u - 2 \nabla^2 u (\nabla u, \cdot) \\
&= 2 u \nabla u - 2 \left( \mathring{\nabla}^2 u + u g \right)(\nabla u, \cdot) \\
&= - 2 \mathring{\nabla}^2 u(\nabla u, \cdot). 
\end{split}
\end{align} 
Since $\mathring{\nabla}^2 u$ has rank 2 and is trace-free, it is easy to see that 
\begin{align*}
|\nabla \eta|^2 = \left| 2 \mathring{\nabla}^2 u(\nabla u, \cdot) \right|^2 = 2|\mathring{\nabla}^2 u|^2 |\nabla u|^2. 
\end{align*}
Since $\nabla u \neq 0$ on $Y \setminus \mathcal{C}$, we have 
\begin{align*}
2|\mathring{\nabla}^2 u|^2 = \dfrac{|\nabla \eta|^2} {|\nabla u|^2}.  
\end{align*}
Substituting this into (\ref{etaPDE1}), we get (\ref{etaPDE2}). 

Finally, to prove (\ref{etaPDE3}) we simply apply the conformal rescaling laws 
\begin{align*}
\Delta_h &= r^2 \Delta_{\bar{h}}, \\
|\nabla_h f|_h^2 &= r^2 |\nabla_{\bar{h}} f|_{\bar{h}}^2, \ \ f \in C^1(Y), \\
|B|^2 &= r^2 |\bar{B}|^2. 
\end{align*}
We remark that by Theorem 3 of \cite{FH}, $|\bar{B}|^2 = O(r^2)$, so the second term on the right-hand side of (\ref{etaPDE3}) is bounded up to  $\partial Y$.
\end{proof} 

We now come to the two main results of this section. 

\begin{proposition} \label{CP1Prop}   Let $Y^2 \subset \mathbb{H}^{n+1}$ be a connected, properly embedded minimal surface satisfying the regularity assumptions of Theorem \ref{Thm1}.  Assume $Y$ is critical in either of the senses of Theorem \ref{Thm1}.   If 
\begin{align} \label{bdymax}
\max_{\partial Y} \eta = \max_{Y} \eta, 
\end{align}
then $Y$ is totally geodesic.  
\end{proposition} 

\begin{proof}  Let $p$ be a point on $\partial Y$ at which $\eta$ attains its maximum on $\bar{Y}$.  As we observed above, $\nabla_{\bar{h}} r \neq 0$ in a collar neighborhood of $\partial Y$.   For $\delta > 0$ small, let $p_{\delta} \in Y$ be the point at distance $\delta$ from $p$, and let $\mathcal{B} = \mathcal{B}_{\delta}(p_{\delta})$ be the geodesic ball centered at $p_{\delta}$ with radius $\delta$.  We assume that $\delta > 0$ is chosen small enough so that $|\nabla_{\bar{h}} r| \neq 0$ in $\mathcal{B}$ and $\mathcal{B} \subset Y$.  On $\mathcal{B}$,  (\ref{etaPDE3}) can be written 
\begin{align} \label{etaPDEZ2} \begin{split} 
\Delta_{\bar{h}} \eta &=  \langle X, \nabla_{\bar{h}} \eta \rangle_{\bar{h}} + r^{-2} |\bar{B}|^2  |\nabla_{\bar{h}} r|_{\bar{h}}^2 \\
&\geq  \langle X, \nabla_{\bar{h}} \eta \rangle_{\bar{h}},
\end{split}
\end{align} 
where 
\begin{align} \label{vex}
X = - \frac{r^2}{|\nabla_{\bar{h}}r|_{\bar{h}}^2} \nabla_{\bar{h}} \eta
\end{align}  
is a bounded vector field in $\mathcal{B}$.  Since $\eta \leq \eta (p)$ in $\mathcal{B}$, it follows from the Hopf boundary point lemma that either $\eta$ is constant in $\mathcal{B}$, or
\begin{align*}
\frac{\partial}{\partial r}\eta(p) < 0. 
\end{align*}
Assuming the latter for a moment, recall by Lemma \ref{betaLemma} that 
\begin{align*}
\frac{\partial}{\partial r}\eta \vert_{r = 0} &= \frac{\partial}{\partial r} \left(  |\gamma''(s)|^2 + O(r^2) \right) \vert_{r=0} \\
&= 0, 
\end{align*}
which is a contradiction.  Therefore, $\eta$ must be constant in $\mathcal{B}$.   But then (\ref{etaPDEZ2}) implies $\bar{B} = 0$ on $\mathcal{B}$ (since $\nabla_{\bar{h}} r \neq 0$ on $\mathcal{B}$).  By conformal invariance, this implies $B$, the trace-free second 
fundamental form of $Y \subset \mathbb{H}^{n+1}$, vanishes on $\mathcal{B}$.  By minimality, this implies the second fundamental form $A$ of $Y$ vanishes on $\mathcal{B}$.   Since minimal surfaces in $\mathbb{H}^{n+1}$ are real analytic (\cite{Leung}, Lemma 1) and $Y$ is connected, it follows that $A \equiv 0$ on $Y$, so $Y$ is totally geodesic. 
\end{proof}

\begin{proposition}  \label{CP2Prop}  Let $Y^2 \subset \mathbb{H}^{n+1}$ be a connected, properly embedded minimal surface satisfying the regularity assumptions of Theorem \ref{Thm1}.  Assume $Y$ is not totally geodesic, and that there is a point $p$ in the interior of $Y$ such that $\eta(p) = \max_{\bar{Y}} \eta$.   Then  \medskip 

\noindent $(i)$  $p \in \mathcal{C}$.   \medskip 

\noindent $(ii)$ $\det \nabla_{\bar{h}}^2 r(p) \leq 0$.
\end{proposition}

\begin{proof} Suppose $p \notin \mathcal{C}$.  Then we can find a small open ball $\mathcal{B} \subset Y$ centered at $p$ such that $\nabla_{\bar{h}} r \neq 0$ in $\mathcal{B}$.  Then  (\ref{etaPDEZ2}) holds in $\mathcal{B}$, with $X$ bounded.   By the strong maximum principle it follows that $\eta$ is constant on $\mathcal{B}$. We can then repeat the argument above to find that $Y$ is totally geodesic.  Therefore, $p \in \mathcal{C}$. 

To prove $(ii)$, choose normal coordinates $\{ x^1, x^2 \}$ in $Y$ centered at $p$, such that $\{ \frac{\partial}{\partial x^1}\vert_p, \frac{\partial}{\partial x^2}\vert_p \}$ diagonalize $\nabla^2_{\bar{h}}r(p)$. Since $p$ is a critical point of $r$, the Taylor expansion of $r$ near $p$ can be written 
\begin{align} \label{rT1}
r = r(p) + \lambda_1 (x^1)^2 + \lambda_2 (x^2)^2 + T_{ijk} x^i x^j x^k +  O(|x|^4),
\end{align}
where $\{ 2 \lambda_1, 2 \lambda_2 \}$ are the eigenvalues of $\nabla_{\bar{h}}^2 r$.  By (\ref{rPDE1}), 
\begin{align} \label{rPDEp} 
\Delta_{\bar{h}} r(p) = - 2 r(p)^{-1} < 0.
\end{align}
In normal coordinates we have 
\begin{align} \label{rT2} \begin{split} 
\Delta_{\bar{h}} r(p) &= \frac{\partial^2 r}{(\partial x^1)^2} \big\vert_{p} + \frac{\partial^2 r}{(\partial x^2)^2} \big\vert_p \\
&= 2 (\lambda_1 + \lambda_2) < 0. 
\end{split} 
\end{align}
Therefore, if we order the eigenvalues of the Hessian so that $\lambda_1 \leq \lambda_2$, then $\lambda_1 < 0$.  Also, (\ref{rPDEp}) and (\ref{rT2}) imply 
\begin{align} \label{rT3} 
r(p) = - (\lambda_1 + \lambda_2)^{-1}. 
\end{align} 

The expansion (\ref{rT1}) along with (\ref{rT3}) imply the following expansion for $\eta$ near $p$:  
\begin{align} \label{eT1} 
\eta = ( \lambda_1 + \lambda_2)^2 \Big\{1 +  2 \lambda_1 \left( \lambda_2 - \lambda_1 \right)(x^1)^2 + 2 \lambda_2 \left( \lambda_1 - \lambda_2 \right) (x^2)^2  + O(|x|^3)  \Big\}.
\end{align} 

We claim that $\lambda_1 < \lambda_2$.  To see this, suppose $\lambda = \lambda_1 = \lambda_2 < 0.$  It follows from (\ref{rT1}) that $p$ is a strict local maximum for $r$, and consequently $p$ is an isolated critical point.  Therefore, the differential inequality (\ref{etaPDEZ2}) holds (in the classical sense) on $\mathcal{B} \setminus \{ p \}$ for some small ball $\mathcal{B}$ containing $p$.  Moreover, near $p$ the expansion for $r$ can be used to write an expansion for the vector field $X$ in (\ref{vex}) on $\mathcal{B} \setminus \{ p \}$.  If we write
\begin{align*}
X = X^1 \frac{\partial}{\partial x^1} + X^2 \frac{\partial}{\partial x^2}, 
\end{align*}
then 
\[
X^1
=
\frac{6}{\lambda}
\frac{
T_{111}(x^1)^2
+
2T_{112}x^1x^2
+
T_{122}(x^2)^2
}{
(x^1)^2+(x^2)^2
}
+
O(|x|)
\]
\[
X^2
=
\frac{6}{\lambda}
\frac{
T_{112}(x^1)^2
+
2T_{122}x^1x^2
+
T_{222}(x^2)^2
}{
(x^1)^2+(x^2)^2
}
+
O(|x|).
\]
In particular, $X$ is bounded near $p$.  We may therefore define $X(p)$ arbitrarily, and since $\eta \in C^2(\mathcal{B})$, we can integrate by parts to show that $\eta$ is a weak (i.e., $W^{1,2}$-) supersolution of  (\ref{etaPDEZ2}) on all of $\mathcal{B}$.  The strong maximum principle for weak solutions (see Theorem 8.20 in \cite{GT}) implies that $\eta$ is constant in $\mathcal{B}$.  We may then argue as above to show that $Y$ is totally geodesic.  Therefore, $\lambda_1 < \lambda_2$.

Since $(\lambda_1 - \lambda_2) < 0$ and $p$ is a maximum point of $\eta$, the expansion (\ref{eT1}) implies 
\begin{align*} 
\lambda_1 &< 0, \\
\lambda_2 &\geq 0. 
\end{align*}
By (\ref{rT1}), we can express the Hessian of $r$ at $p$ by 
\begin{align} \label{Hessr} 
\nabla_{\bar{h}}^2 r(p) = \begin{bmatrix}
2\lambda_1 & 0 \\
0 & 2\lambda_2
\end{bmatrix}, 
\end{align} 
hence 
\begin{align}
\det \nabla_{\bar{h}}^2 r(p) = 4 \lambda_1 \lambda_2 \leq 0,
\end{align}
which shows that $(ii)$ holds.  
\end{proof}

\bigskip 

Using Propositions \ref{CP1Prop} and \ref{CP2Prop}, we can now give the proof of Theorem \ref{Thm1}:   \medskip 

\begin{proof}[Proof of Theorem \ref{Thm1}] 

Suppose $(i)$ or $(ii)$ holds.  If (\ref{bdymax}) holds, then Proposition \ref{CP1Prop} implies that $Y$ is totally geodesic, and we are done.  Therefore, we may assume the maximum of $\eta$ occurs at an interior point $p \in Y$, and that $Y$ is not totally geodesic.   By Proposition \ref{CP2Prop}, we know that $p \in \mathcal{C}$, and 
\begin{align} \label{negdet}
\det \nabla_{\bar{h}}^2 r(p) \leq 0. 
\end{align}

Since (\ref{Asmall}) holds, 
\begin{align} \label{ABsmall}
|A|^2 = |B|^2 < 2. 
\end{align}
By conformal invariance, this implies 
\begin{align}  \label{Brbound}
|\bar{B}|^2 < 2 r^{-2}.
\end{align}
If $\bar{A}$ denotes the second fundamental form of $Y$ with respect to $g_0$, then 
\begin{align} \label{Ar}
\nabla_{g_0}^2 r = \nabla_{\bar{h}}^2 r - \langle \bar{A}, \nabla_{g_0}^{\perp}r \rangle, 
\end{align}
where $\nabla_{g_0}^{\perp} r$ is the normal component of the  (Euclidean) gradient of $r$.  Since the left-hand side vanishes, by taking the trace-free part of this identity we get 
\begin{align} \label{Br2} 
\mathring{\nabla}_{\bar{h}}^2 r = \langle \bar{B}, \nabla_{g_0}^{\perp} r \rangle. 
\end{align} 
By Cauchy-Schwarz, 
\begin{align} \label{Br3} 
|\mathring{\nabla}_{\bar{h}}^2 r|_{\bar{h}}^2 \leq  | \bar{B} |^2 |\nabla_{g_0}^{\perp} r |_{g_0^{\perp}}^2.  
\end{align} 
Since $|\nabla_{g_0} r|^2 = 1,$ it follows that 
\begin{align*}
|\nabla_{g_0}^{\perp} r|_{g_0^{\perp}}^2 = 1 - |\nabla_{\bar{h}}r|_{\bar{h}}^2.
\end{align*}
Substituting this into (\ref{Br3}) and using (\ref{Brbound}), we obtain 
\begin{align} \label{tfr}
|\mathring{\nabla}_{\bar{h}}^2 r|_{\bar{h}}^2 < 2 r^{-2} \left( 1 - |\nabla_{\bar{h}} r|_{\bar{h}}^2 \right),  
\end{align}
which holds at any point where $|\nabla_{\bar{h}}r|_{\bar{h}} < 1$. In particular, (\ref{tfr}) holds at $p$, since $p \in \mathcal{C}$.  
By (\ref{rPDE1}), we can also write
\begin{align} \label{tfr2} 
|\mathring{\nabla}_{\bar{h}}^2 r|_{\bar{h}}^2 < r^{-1} \left( - \Delta_{\bar{h}} r \right).   
\end{align}

Since $p \in \mathcal{C}$, we can also use the Taylor expansion (\ref{rT1}) of $r$ near $p$.  Recall from (\ref{rT2}) and (\ref{rT3}) that 
\begin{align*}
- \Delta_{\bar{h}} r(p) &= - 2 ( \lambda_1 + \lambda_2), \\
r(p) &= - (\lambda_1 + \lambda_2)^{-1}, 
\end{align*}
hence 
\begin{align} \label{rT123}
r^{-1} \left( - \Delta_{\bar{h}} r \right) = 2 ( \lambda_1 + \lambda_2)^2.
\end{align}
Also, by (\ref{Hessr}), 
\begin{align} \label{Hessr2} 
\mathring{\nabla}_{\bar{h}}^2 r(p) = \begin{bmatrix}
\lambda_1 - \lambda_2 & 0 \\
0 & \lambda_2 - \lambda_1
\end{bmatrix},
\end{align} 
which implies 
\begin{align} \label{TH}
|\mathring{\nabla}_{\bar{h}}^2 r|_{\bar{h}}^2 = 2 (\lambda_1 - \lambda_2)^2. 
\end{align} 
Consequently, at $p$ (\ref{tfr2}) is equivalent to 
\begin{align*}
2 (\lambda_1 - \lambda_2)^2 < 2 ( \lambda_1 + \lambda_2)^2,
\end{align*} 
which implies
\begin{align*}
4 \lambda_1 \lambda_2 = \det \nabla_{\bar{h}}^2 r > 0.
\end{align*}
Since this contradicts (\ref{negdet}), we conclude that $Y$ is totally geodesic.  
\end{proof} 

\bigskip

\section{$\epsilon$-regularity}  \label{SecER} 

In this section we prove Theorems \ref{Thm2} and \ref{Thm3}.

Let $Y \subset \mathbb{H}^{n+1}$ be a connected, properly embedded minimal surface satisfying the regularity assumptions of Theorem \ref{Thm1}.  We will show that if $\epsilon_0 > 0$ is small enough, then $|A|^2 = |B|^2 < 2$, and the result follows from Theorem \ref{Thm1}.    

We first observe that by the renormalized area formula of Alexakis-Mazzeo (see Proposition 3.1 of \cite{AM}), (\ref{gap}) implies 
\begin{align} \label{Bgap} 
\int_Y |B|^2 \, dA_h < 2 \epsilon_0.
\end{align}
Since $Y \subset \mathbb{H}^{n+1}$ is minimal, the Simons identity implies that $|B|^2$ satisfies the following differential inequality (see \cite{CdCK}, \cite{LiLi}): 
\begin{align} \label{sid1} 
\frac{1}{2} \Delta |B|^2 \geq |\nabla B|^2 - 2 |B|^2 - \frac{3}{2}|B|^4, 
\end{align} 
where from now on all covariant derivatives and norms are with respect to $h$.  We will use a standard Moser iteration argument to show that the $L^2$-bound (\ref{Bgap}) can be converted to a pointwise bound. 

Since we will be using Moser iteration, we will need control of the Sobolev constant.  The minimality of $Y$ implies the Gauss curvature $K \leq -1$.  Also, by Bernstein's inequality (\ref{BernI}) the assumption $\mathcal{A}(Y)>-2\pi\chi(Y)-\epsilon_0$ implies
\[
-2\pi\chi(Y)-\epsilon_0<-2\pi.
\]
If \(\epsilon_0<2\pi\), then \(\chi(Y)>0\). Since \(\partial Y=\gamma\) is connected,
it follows that \(Y\) is topologically a disk.  By a result of Chavel-Feldman \cite{chavel1980isoperimetric}, we then have a (sharp) bound for the isoperimetric constant of $Y$.  Using this bound we obtain, by standard arguments via the coarea formula, a Sobolev inequality of the form
\begin{align} \label{CS} 
\left( \int \phi^2 \, dA_h \right)^{1/2} \leq \frac{1}{\sqrt{4\pi}}\int |\nabla \phi| \, dA_h,
\end{align}
which holds for all $\phi \in C^1_0(Y)$.  

Since the Simons identity gives a nonlinear PDE for $B$, we need a 'supercritical' $L^p$-estimate for $B$.

\begin{lemma}\label{L4Lemma}
Assume that the Sobolev inequality (\ref{CS}) holds for all \(\varphi\in C^1_0(Y)\). There are universal constants
\(\epsilon_1>0\) and \(C_1>0\) such that if
\[
\int |B|^2\,dA < \epsilon_1,
\]
then
\[
\int |B|^4\,dA \leq C_1,
\]
where $dA$ is the area form of $h$.  In fact, one may take \(\epsilon_1=\pi/12\) and \(C_1=\pi/9\).
\end{lemma}

\begin{proof}  By conformal invariance it is easy to see that $|B| \in L^p(Y)$ for any $p > 1$.  Set
\[
f:=|B|,
\]
and assume 
\[ \int f^2 \, dA < \epsilon_1. \]  
Choose a standard cutoff function $\chi_R\in C^\infty_0(Y)$ such that
\[
0\le \chi_R\le 1,\qquad
\chi_R\equiv 1 \ \text{on } B_R,\qquad
\operatorname{supp}\chi_R\subset B_{2R},\qquad
|\nabla \chi_R|\le \frac{2}{R}.
\]
By the Simons inequality (\ref{sid1}) and Kato's inequality, 
\[
\frac12 \Delta(f^2)=f\Delta f+|\nabla f|^2
\ge |\nabla f|^2-2f^2-\frac32 f^4,
\]
hence
\begin{align} \label{sidf} 
f\Delta f \ge -2f^2-\frac32 f^4.
\end{align} 
Multiplying by $\chi_R^2$ and integrating by parts gives
\[
\int \chi_R^2 |\nabla f|^2\,dA
+2\int \chi_R f\langle \nabla f,\nabla \chi_R\rangle\,dA
\le
2\int f^2\chi_R^2\,dA +\frac32\int f^4\chi_R^2\,dA .
\]
Estimating the cross term by
\[
2\big|\chi_R f\langle \nabla f,\nabla \chi_R\rangle\big|
\le \frac12 \chi_R^2 |\nabla f|^2 + 2 f^2 |\nabla \chi_R|^2,
\]
we obtain
\begin{align} \label{L41}
\int \chi_R^2 |\nabla f|^2\,dA
\le
4\int f^2\chi_R^2\,dA
+4\int f^2|\nabla\chi_R|^2\,dA
+3\int f^4\chi_R^2\,dA.
\end{align} 

Applying the Sobolev inequality (\ref{CS}) to $\varphi=(f\chi_R)^2$, we get 
\[
\left(\int f^4\chi_R^4\,dA\right)^{1/2}
\le
\frac1{\sqrt{4\pi}}
\int \big|\nabla((f\chi_R)^2)\big|\,dA
=
\frac1{\sqrt{\pi}}
\int |f\chi_R|\,|\nabla(f\chi_R)|\,dA.
\]
By Cauchy--Schwarz,
\begin{align} \label{L42} 
\int f^4\chi_R^4\,dA
\le
\frac1{\pi}
\left(\int f^2\chi_R^2\,dA\right)
\left(\int |\nabla(f\chi_R)|^2\,dA\right).
\end{align} 
Since $\chi_R\le 1$ and $\int f^2\,dA <\epsilon_1$, we have
\[
\int f^2\chi_R^2\,dA \le \epsilon_1.
\]
Also,
\[
|\nabla(f\chi_R)|^2 \le 2\chi_R^2 |\nabla f|^2 + 2 f^2 |\nabla \chi_R|^2,
\]
hence by (\ref{L41}) 
\[
\int |\nabla(f\chi_R)|^2\,dA
\le
8\int f^2\chi_R^2\,dA
+10\int f^2|\nabla\chi_R|^2\,dA
+6\int f^4\chi_R^2\,dA.
\]
Substituting this into (\ref{L42}), we get
\[
\int f^4\chi_R^4\,dA
\le
\frac{\epsilon_1}{\pi}
\left(
8\int f^2\chi_R^2\,dA
+10\int f^2|\nabla\chi_R|^2\,dA
+6\int f^4\chi_R^2\,dA
\right).
\]

Letting $R\to\infty$, we have 
\[
\int f^4\,dA
\le
\frac{\epsilon_1}{\pi}
\left(
8\int f^2\,dA + 6\int f^4\,dA
\right),
\]
hence 
\[
\left(1-\frac{6\epsilon_1}{\pi}\right)\int f^4\,dA
\le \frac{8\epsilon_1^2}{\pi}.
\]
Consequently, if $\epsilon_1<\pi/6$ then 
\[
\int |B|^4\,dA = \int f^4\,dA
\le
\frac{8\epsilon_1^2}{\pi-6\epsilon_1}.
\]
\end{proof}

We are now ready to state the Moser iteration result.  Since we could not find an easily adaptable version to our setting, we also provide a proof.

\begin{lemma} \label{MoserLemma}   Assume that the Sobolev inequality (\ref{CS}) holds for all \(\varphi\in C^1_0(Y)\). Then there is a universal constant
\(\epsilon_2>0\) such that if
\[
\int |B|^2\,dA<2\epsilon_2,
\]
then
\[
|B|^2<2
\]
on \(Y\).
\end{lemma}

\begin{proof}  For $\varphi \in C_0^\infty(Y)$, by the Sobolev inequality and Hölder's inequality,
\begin{align*}
\left(\int \varphi^6\, dA\right)^{1/2} &\le \frac{1}{\sqrt{4\pi}} \int |\nabla \varphi^3|\, dA \\
&= \frac{3}{\sqrt{4\pi}} \int \varphi^2 |\nabla \varphi|\, dA \\
&\le \frac{3}{\sqrt{4\pi}}
\left(\int |\nabla \varphi|^2\, dA\right)^{1/2}
\left(\int \varphi^4\, dA\right)^{1/2} \\
&\le \frac{3}{\sqrt{4\pi}}
\left(\int |\nabla \varphi|^2\, dA\right)^{1/2}
\left(\int \varphi^6\, dA\right)^{1/4}
\left(\int \varphi^2\, dA\right)^{1/4}.
\end{align*}
Therefore, 
\begin{align} \label{S6} \begin{split} 
\left(\int \varphi^6\, dA\right)^{1/3} &\le
\left(\frac{3}{\sqrt{4\pi}}\right)^{4/3}
\left(\int |\nabla \varphi|^2\, dA\right)^{2/3}
\left(\int \varphi^2\, dA\right)^{1/3} \\
&\le \int |\nabla \varphi|^2\, dA
+
\frac{3}{4\pi^2}\int \varphi^2\, dA,
\end{split}
\end{align}
which holds for any $\varphi \in W_0^{1,2}(Y)$.

As in the proof of Lemma \ref{L4Lemma}, let $f = |B|$. Fix $p \in Y$, and for $R \gg 1$ let $\eta$ be a $C^\infty$-cutoff with
$\eta \equiv 1$ on $B_R(p)$, $\eta \equiv 0$ on $Y \setminus B_{2R}(p)$, $0 \le \eta \le 1$, and $|\nabla \eta| \le 2R^{-1}$,
where $B_R(p)$ is the geodesic ball centered at $p$ with radius $R$.  By (\ref{sidf}),
\[
\Delta f \ge -2f - \frac{3}{2}f^3.
\]
For $m \geq 1$, we multiply both sides of the above by $\eta^{2m}f^{2m-1}$ and integrate by parts: 
\begin{align*}
& \int \eta^{2m} f^{2m-1} (-\Delta f)\, dA \le \int 2\eta^{2m} f^{2m}\, dA
+
\int \frac{3}{2} f^{2m+2}\eta^{2m}\, dA \\
& \hskip-.25in  \ \Longrightarrow \\ 
& \int (2m-1)\eta^{2m} f^{2m-2} |\nabla f|^2\, dA
+
2m \int \eta^{2m-1} f^{2m-1} \langle \nabla \eta,\nabla f\rangle\, dA \\
& \hskip.5in  \le \int 2\eta^{2m} f^{2m}\, dA + \int \frac{3}{2} f^{2m+2}\eta^{2m}\, dA,
\end{align*}
which implies
\[
\int \eta^{2m} f^{2m-2} |\nabla f|^2\, dA
\le
\int -\frac{2m}{2m-1}\eta^{2m-1} f^{2m-1}
\langle \nabla \eta,\nabla f\rangle\, dA
\]
\[
\qquad
+
\int \frac{2}{2m-1}\eta^{2m} f^{2m}\, dA
+
\int \frac{3}{2(2m-1)} f^{2m+2}\eta^{2m}\, dA.
\]
We can estimate the cross term by
\begin{align*}
\int -\frac{2m}{2m-1}\eta^{2m-1} f^{2m-1}
\langle \nabla \eta,\nabla f\rangle\, dA &\le \int \frac{1}{2}\eta^{2m} f^{2m-2} |\nabla f|^2\, dA \\
& \ \ \ \ \ \ + \int \frac{2m^2}{(2m-1)^2}\eta^{2m-2} f^{2m} |\nabla \eta|^2\, dA.
\end{align*}
Therefore,
\begin{align} \label{Mi2} \begin{split}
\int \eta^{2m} f^{2m-2} |\nabla f|^2\, dA
&\le \int \frac{4m^2}{(2m-1)^2}\eta^{2m-2} f^{2m} |\nabla \eta|^2\, dA \\
& \ \ \ \ + \int \frac{4}{2m-1}\eta^{2m} f^{2m}\, dA + \int \frac{3}{2m-1} f^{2m+2}\eta^{2m}\, dA.
\end{split}
\end{align} 

Now,
\[
\nabla(\eta^m f^m)
=
m\eta^{m-1}f^m \nabla \eta
+
m\eta^m f^{m-1}\nabla f,
\]
hence
\[
|\nabla(\eta^m f^m)|^2
=
m^2\eta^{2m-2}f^{2m}|\nabla \eta|^2
+
2m^2\eta^{2m-1}f^{2m-1}\langle \nabla \eta,\nabla f\rangle
+
m^2\eta^{2m}f^{2m-2}|\nabla f|^2
\]
\[
\le
2m^2\eta^{2m-2}f^{2m}|\nabla \eta|^2
+
2m^2\eta^{2m}f^{2m-2}|\nabla f|^2.
\]
Combining with (\ref{Mi2}), we find 
\begin{align} \label{Mi3} \begin{split} 
\int |\nabla(\eta^m f^m)|^2\, dA &\le \int \left(\frac{8m^4}{(2m-1)^2} +
2m^2
\right)
\eta^{2m-2}f^{2m}|\nabla \eta|^2\, dA  \\
& \ \ \ \ + \int \frac{8m^2}{2m-1}\eta^{2m}f^{2m}\, dA
+
\int \frac{6m^2}{2m-1}f^{2m+2}\eta^{2m}\, dA.
\end{split}
\end{align}
By the Sobolev inequality (\ref{S6}), 
\[
\left(\int (\eta^m f^m)^6\, dA\right)^{1/3}
\le
\int |\nabla(\eta^m f^m)|^2\, dA
+
\int \frac{3}{4\pi^2}\eta^{2m}f^{2m}\, dA,
\]
hence by (\ref{Mi3}) 
\begin{align} \label{Mi4} \begin{split} 
\left(\int (\eta^m f^m)^6\, dA\right)^{1/3} &\le
\int
\left(
\frac{8m^4}{(2m-1)^2}
+
2m^2
\right)
\eta^{2m-2}f^{2m}|\nabla \eta|^2\, dA \\
& \ \ \ \ + \int
\left(
\frac{8m^2}{2m-1}
+
\frac{3}{4\pi^2}
\right)
\eta^{2m}f^{2m}\, dA
+
\int \frac{6m^2}{2m-1}f^{2m+2}\eta^{2m}\, dA.
\end{split}
\end{align} 

Using Hölder's inequality we can estimate the last term in (\ref{Mi4}) by 
\[
\int \frac{6m^2}{2m-1}f^{2m+2}\eta^{2m}\, dA
\le
\frac{6m^2}{2m-1}
\left(\int f^4\, dA\right)^{1/2}
\left(\int f^{4m}\eta^{4m}\, dA\right)^{1/2}
\]
\[
\le
\frac{6m^2\|f\|_{L^4}^2}{2m-1}
\left(\int (\eta^m f^m)^6\, dA\right)^{1/4}
\left(\int \eta^{2m}f^{2m}\, dA\right)^{1/4}
\]
\[
\le
\frac{1}{2}
\left(\int (\eta^m f^m)^6\, dA\right)^{1/3}
+
\left(\frac{27}{32}\right)
\left[
\left(\frac{6m^2}{2m-1}\right)\|f\|_{L^4}^2
\right]^4
\int \eta^{2m}f^{2m}\, dA.
\]

Substituting this into (\ref{Mi4}), we find
\begin{align} \label{Mi4b} 
\left(\int (\eta^m f^m)^6\, dA\right)^{1/3}
\le
C_1(m)
\int \eta^{2m-2}f^{2m}|\nabla \eta|^2\, dA
+
C_2(m)
\int \eta^{2m}f^{2m}\, dA,
\end{align}
where
\[
C_1(m)
=
\frac{16m^4}{(2m-1)^2}
+
4m^2,
\]
\begin{align} \label{C2m} 
C_2(m)
=
\frac{8m^2}{2m-1}
+
\frac{3}{4\pi^2}
+
\left(\frac{27}{16}\right)
\left[
\left(\frac{6m^2}{2m-1}\right)\|f\|_{L^4}^2
\right]^4.
\end{align} 
Since $m \ge 1$, it follows that $2m-1 \ge m$.
Then using Lemma \ref{L4Lemma} (with $\epsilon_1 = \pi/12$) we can estimate $C_2$ by
\[
C_2(m)
\le
8m
+
27\pi^2 m^4
+
\frac{3}{4\pi^2}
\]
\[
\le
\overline{C}_2(1+m^4),
\]
where $\overline{C}_2$ is an absolute constant which does not depend on $m.$
Therefore, by (\ref{Mi4b}), 
\begin{align} \label{Mi5} 
\left(\int (\eta^m f^m)^6\, dA\right)^{1/3}
\le
C_1(m)
\int \eta^{2m-2}f^{2m}|\nabla \eta|^2\, dA
+
\overline{C}_2(1+m^4)
\int \eta^{2m}f^{2m}\, dA.
\end{align}

If we take $m=1$ in (\ref{Mi5}) then let $R \to \infty$, we obtain
\[
\|f\|_{L^6}
\le
\overline{C}_2^{1/2}2^{1/2}\|f\|_{L^2}.
\]
We can now iterate in the usual manner, each time applying (\ref{Mi5}) and letting $R \to \infty$.
If we take $m_k = 3^k$, then 
\[
\|f\|_{L^{2m_k}}
\le
A_k\|f\|_{L^2},
\]
where
\[
A_k
=
\overline{C}_2^{\frac12\left[\frac{1}{m_{k-1}}+\cdots+1\right]}
\prod_{j=1}^{k-1} (1+m_j^4)^{\frac{1}{2m_j}}.
\]
Letting $k \to \infty$, we conclude
\[
\|B\|_{L^\infty}
\le
A_{\infty}\,\|B\|_{L^2}
<
A_{\infty}\, \sqrt{2 \epsilon_2},
\]
where $A_{\infty} = \lim_{k \to \infty} A_k$. In particular, once $\epsilon_2>0$ is small enough, $|B|^2 <2$.
\end{proof}

\begin{proof}[Proof of Theorem \ref{Thm2}]
Let \(\epsilon_1\) be the constant in Lemma \ref{L4Lemma} and let \(\epsilon_2\) be the
constant in Lemma \ref{MoserLemma}. Define
\begin{align} \label{ez}
\epsilon_0
=
\frac12\min\left\{2\pi,\frac{\epsilon_1}{2},\epsilon_2\right\}.
\end{align}
Then the assumption
\[
\mathcal A(Y)>-2\pi\chi(Y)-\epsilon_0
\]
implies, by the discussion preceding Lemma \ref{L4Lemma}, that \(Y\) is a disk and hence
the Sobolev inequality \((4.3)\) holds. Also, by the renormalized area formula,
\[
\int |B|^2\,dA_h<2\epsilon_0.
\]
Since \(2\epsilon_0<2\epsilon_2\), Lemma \ref{MoserLemma} gives \(|B|^2<2\) on \(Y\). Since
\(Y\) is minimal, \(A=B\). Finally, in case $(i)$ Corollary \ref{AHCor}  gives \(U_3=0\), while in case $(ii)$
Proposition \ref{prop:nondegenerate-critical-implies-U3} gives \(U_3=0\). Therefore, Theorem \ref{Thm1} implies 
 \(\gamma=\partial Y\) is a circle and that \(Y\) is a totally geodesic
disk.
\end{proof}

\medskip 

\begin{remark} \label{epRemark}  A rough estimate of the constant $A_{\infty}$ shows that one can take $\epsilon_0 \approx 10^{-6}$.  In particular, $\epsilon_0$ does not depend on the codimension.  
\end{remark}

\medskip

\begin{proof}[Proof of Theorem \ref{Thm3}]   Suppose either (\ref{Asmall}) or (\ref{gap}) holds, where $\epsilon_0$ is given by (\ref{ez}).  Then (as we saw in the proof of Theorem \ref{Thm2} above) in both cases (\ref{Asmall}) is valid.  

Since the maximum of $|\bar{H}| = 2\sqrt{\eta}$ is attained at an interior point $p \in Y$, Proposition \ref{CP2Prop} implies that either $Y$ is totally geodesic, or  $p \in \mathcal{C}$ and (\ref{negdet}) holds.  However, as we saw in the proof of Theorem \ref{Thm2}, (\ref{Asmall}) contradicts (\ref{negdet}). 
 \end{proof}

\bigskip

%
%

%
%
%
%

\end{document}